\documentclass[twoside,11pt]{article}

\usepackage{amssymb,amsthm}
\usepackage[colorlinks,urlcolor=green]{hyperref}
\usepackage{graphicx}
\usepackage{multirow}
\usepackage{float}
\usepackage{caption}
\usepackage{epstopdf}
\usepackage{enumerate}
\usepackage{array}
\usepackage{booktabs}
\usepackage{color}
\usepackage{mathrsfs}   
\usepackage{amsmath} 
\usepackage{amsfonts} 
\usepackage{amssymb} 
\usepackage{hyperref}
\usepackage{eucal}
\usepackage{epstopdf}

\newtheorem{theorem}{Theorem}[section]

\newtheorem{lemma}[theorem]{Lemma}

\theoremstyle{definition}
\theoremstyle{plain}

\theoremstyle{plain}

\numberwithin{equation}{section}
\begin{document}
	\pagestyle{myheadings}
	\begin{titlepage}
		\title{\bf Hamiltonian laceability with set of faulty edges in hypercubes}
		\author{Abid Ali {\small and} Weihua Yang$^{*}$ \\
			{\it \small College of Mathematics, Taiyuan University of Technology, Taiyuan, Shanxi, China, 030024}\\
			{\it \small Emails:\ }
			{\it \small lian0003@link.tyut.edu.cn}\\
		{\it \small yangweihua@tyut.edu.cn}
		} 
	\end{titlepage}
	\date{}
	\maketitle \vskip 0.1cm \centerline
	\hrulefill
\begin{abstract}	
Faulty networks are useful because link or node faults can occur in a network. This paper examines the Hamiltonian properties of hypercubes under certain conditional faulty edges. Let consider the hypercube \( Q_n \), for \( n \geq 5 \) and set of faulty edges \( F \) such that \( |F| \leq 4n - 17 \). We prove that a Hamiltonian path exists connecting any two vertices in \( Q_n - F \) from distinct partite sets if they verify the next two conditions: (i) in $Q_n - F$ any vertex has a degree at least 2, and  (ii) in $Q_n - F$ at most one vertex has a degree exactly equal to 2. These findings provide an understanding of fault-tolerant properties in hypercube networks.
\end{abstract}
	\noindent \emph{\bf Keywords}:~{Hamiltonian laceable, Hamiltonian path,  Fault tolerant, Hypercube network}.
	
	\maketitle
	\hrulefill
	
\section{Introduction}
A graph, \( G = (V(G), E(G)) \) models the topology of an interconnection network, where \( V(G) \), is the vertex set corresponds to the processors and \( E(G) \), is the edge set denotes the connections between these processors. Investigating the properties of \( G \) is necessary for designing an efficient network topology. Paths and cycle structures are basic topologies in parallel and distributed systems,  which are especially well-suited for local area networks and for developing parallel algorithms that minimize communications costs \cite{new1}.

In a graph \( G \), a path (respectively, cycle)  is classified as a Hamiltonian path (respectively, cycle) if every vertex of \( G \) appears in path one and only one time. Particularly, a Hamiltonian path visits every vertex only once without forming a cycle, whereas a Hamiltonian cycle returns to its starting vertex. A graph is Hamiltonian if it has a Hamiltonian cycle. Furthermore, if a Hamiltonian path exists between every two distinct vertices in the graph, the graph is considered Hamiltonian-connected. A bipartite graph, \( G = (V_0 \cup V_1, E) \) is one in which the vertex set \( V(G) \) can be divided into two disjoint subsets, \( V_0 \) and \( V_1 \), such that each edge in \( G \) connects a vertex from \( V_0 \) to a vertex in \( V_1 \). A fundamental question in studying the Hamiltonian properties of a graph is determining whether it is Hamiltonian-connected or Hamiltonian. Although in Hamiltonian bipartite graphs  \( |V_0| = |V_1| \), as a Hamiltonian cycle must alternate between vertices in \( V_0 \) and \( V_1 \). Consequently, not all Hamiltonian bipartite graphs are Hamiltonian-connected. For those Hamiltonian bipartite graphs, the concept of Hamiltonian laceability was presented by Simmons \cite{new2}. A bipartite graph \( G = (V_0 \cup V_1, E) \), if connecting any two vertices \( x \) and \( y \) such that \( y \in V_1 \) and \( x \in V_0 \), is Hamiltonian laceable.

The hypercube graph $Q_n$ is an efficient and most commonly used interconnection network. Since for every $n \geq 2$, hypercube $Q_n$ has a Hamiltonian cycle \cite{1} and for every $n \geq 1$ is Hamiltonian laceable \cite{H84}. For the stability of a network, fault tolerance is a crucial index. Faulty networks are useful because link faults or node faults may appear in networks \cite{new3,new4,new5,new6,new7,new8, new11}.

In the hypercube \( Q_n \) with $F$ is a set of faulty edges and $x, y$ is a pair of vertices,  whether a Hamiltonian cycle or (Hamiltonian path) exists in  \( Q_n - F \) connecting vertices \( y \) and \( x \) \cite{new9}? For \( n \geq 3 \), Chan and Lee \cite{CM91} verified that a fault-free Hamiltonian cycle exists if every vertex is connected to at least two non-faulty edges, even when faulty edges are up to \( (2n - 5) \). For a hypercube, the graph depicted in Figure \ref{Fig.8}, this result is optimal since the subgraph contains no fault-free Hamiltonian cycle. Further this result was extended by Liu and Wang \cite{new13}, and proved that for \( n \geq 5 \) with \( |F| \leq 3n - 8 \), under two conditions: (i) Every vertex in \( Q_n - F \) must have a degree at least 2, and (ii) no two non-adjacent vertices within a 4-cycle can both have a degree of 2. Moreover, Lee et al. \cite{new10} derived that if every vertex in \( Q_n-F \) has a degree of at least three, then the graph \( Q_n - F \) has a Hamiltonian cycle.
\begin{figure}[H]
	\centering
	\includegraphics[width=0.45\linewidth, height=0.27\textheight]{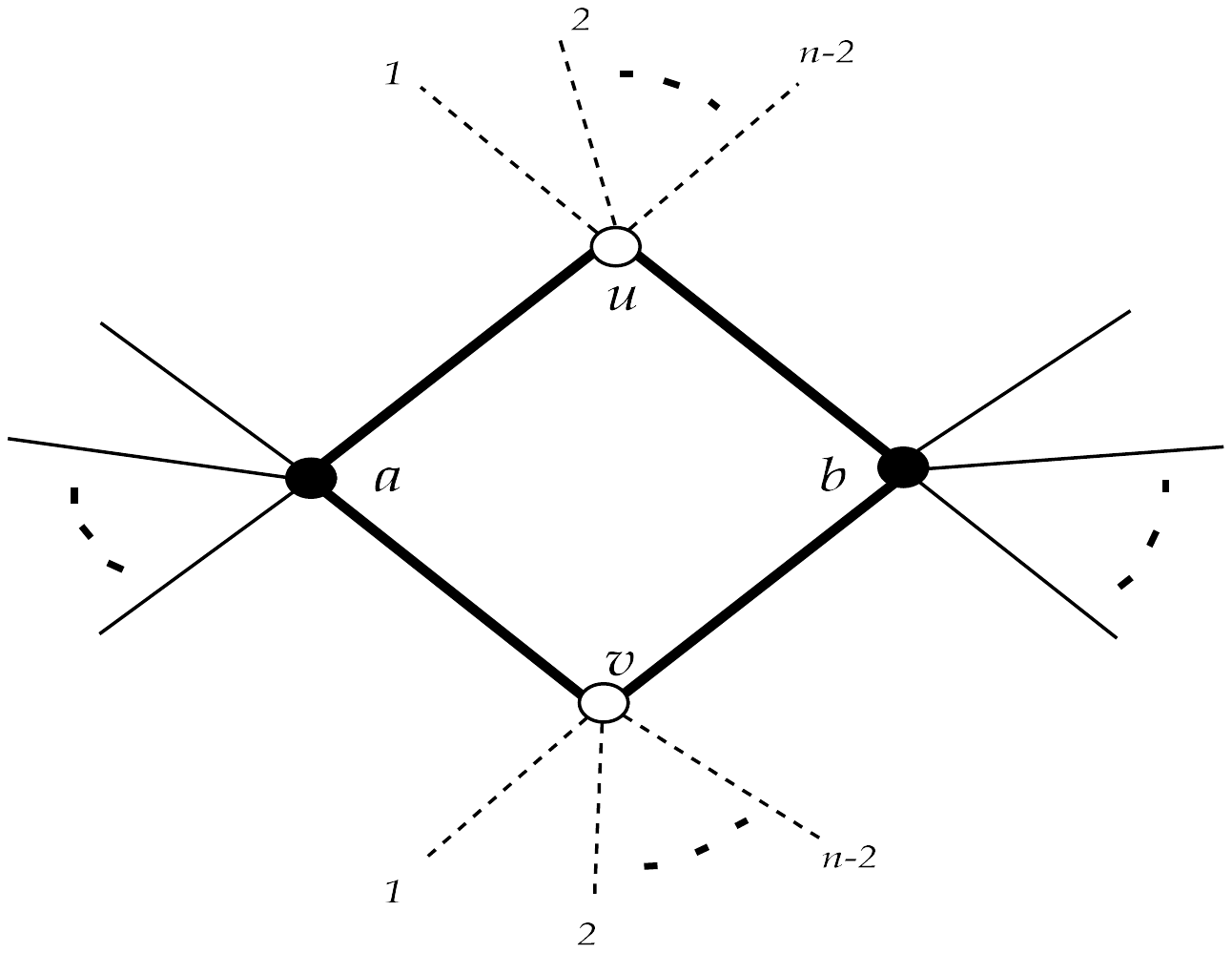}
	\caption[A]{The dashed lines represents faulty edges in forbidden subgraph.}\label{Fig.8}
\end{figure}

Tsai \cite{new12} studied the existence of a Hamiltonian path connecting any two vertices from distinct partite sets in a faulty hypercube under the case that the faulty edges \( |F| \leq 2n - 5 \). Note that if the hypercube \( Q_n \) contains a subgraph similar to those depicted in Figure \ref{Fig.9}(a)-(c), then no Hamiltonian path between vertices \( y \) and \( x \) exists in \( Q_n - F \). Later, Wang and Zhang \cite{new11} extended this study by considering hypercubes with conditionally faulty edges and derived the next result.
\begin{figure}[H]
	\centering
	\includegraphics[width=0.9\linewidth, height=0.24\textheight]{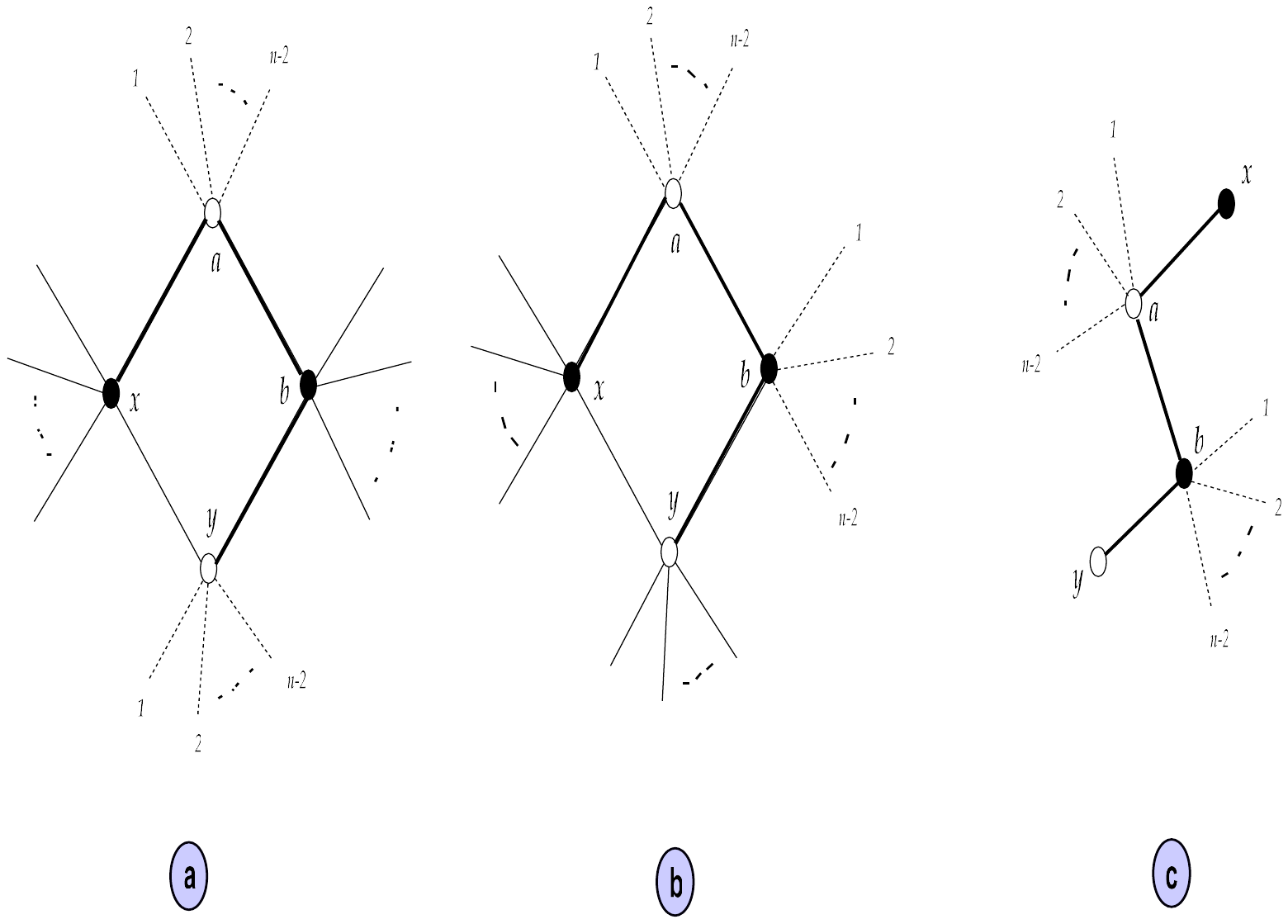}
	\caption[A]{The dashed lines represents faulty edges in forbidden subgraph.}\label{Fig.9}
\end{figure}
\begin{theorem}[\cite{new11}] \label{n1}
Consider faulty edge set \( F \) in \( Q_n \) for \( n \geq 4 \), such hat \( |F| \leq 3n - 11 \). Then, the graph is Hamiltonian laceable when, \( \delta(Q_n - F) \geq 2 \), and contain one vertex at most in \( Q_n - F \) having degree exactly  \( 2 \).
\end{theorem}
In this paper, we show that the result can be extended up to \( 4n - 17 \) faulty edges for \( n \geq 5 \),  a Hamiltonian path exists in  \( Q_n - F \), between every two vertices from distinct partite sets, and holds the below two conditions.
 
(i) Each vertex in \( Q_n - F \) has a degree of at least 2.
 
(ii) There is at most one vertex in \( Q_n - F \) with a degree of 2, and gives the proof of the following theorem.
 \begin{theorem}\label{main}
For \( n \geq 5 \), consider $F$ is faulty edges set in \( Q_n \) such that \( |F| \leq 4n - 17 \).  When, \( \delta(Q_n - F) \geq 2 \), one vertex is at most exists in \( Q_n - F \) with degree 2, then the graph is Hamiltonian laceable.
 \end{theorem}
The structure of this paper is arranged as follows. Some results used in the proof of our problem are provided in Section \ref{sec2}, while Section \ref{sec3} gives the proof of our problem.
\section{Preliminaries and lemmas}\label{sec2}
The notation and terminology are defined below but undefine can be seen in \cite{bondy}. Let a graph \( G \), the edge and vertex set are represented as \( E(G) \) and \( V(G) \), respectively. Consider  \( H_1 \) and \( H_2 \) are two subgraphsof \( G \). The \( H_1 + H_2 \) is used to represent the graph with edge set \( E(H_1) \cup E(H_2) \) and vertex set \( V(H_1) \cup V(H_2) \).  The \( V(F) \) means a set of vertices incident to the edges in \( F \), where a subset \( F \subseteq E(G) \). The graph  \( H + F \) with vertex set \( V(H) \cup V(F) \) and edge set \(E(H) \cup F \). When removing the edges set  \( F \subseteq E(G) \)  from \( G \), the resultant graph is denoted as \( G - F\). Similarly, a subset \( S \subseteq V(G) \), \( G - S \) describes the graph obtained by deleting all vertices in \( S \) and all edges incident to those vertices. When  \( F = \{e\} \) or \( S = \{x\} \), the notations \( G - F \), \( G - S \),  \( H + F \),  and \( V(F) \) are simplified to  \( G - e \), \( G - x \), \( H + e \),  and \( V(e) \), respectively.

The hypercube graph $Q_n$ where $[n]=\{1, 2, \dots, n\}$  with vertex set $V(Q_n) = \{ v : v = v^1 \dots v^n\}$  and  $v^i \in \{0, 1\}$ for $i \in [n]\}$ and edge set $E(Q_n) = \{xy : |\Delta(x, y)| = 1 \}$, where $\Delta(x, y) = \{ i \in [n] : x^i \neq y^i\}$. The $dim(xy)$ is dimension of an edge $xy \in E(Q_n)$ is the integer $j$ with $\Delta(x, y) = \{ j \}$. The exact $j$-dimension edges form a layer, represented as $E_j$, in $Q_n $. Note that the edge set $E(Q_n)$ is divided into $n$ layers, each having $2^{n-1}$ edges.

The parity $p(u)$ of a vertex $u$ in $Q_n$ is described by $p(u) = \sum_{i=1}^n u^i(\text{mod 2})$. Hence, there are \( 2^{n-1} \) vertices of 0 parity and \( 2^{n-1} \) vertices of 1 parity  in \( Q_n \). The vertices with parity 0 and 1  are white and black, respectively. Since \( Q_n \) is a bipartite graph, the color classes form a bipartition of \( Q_n \), meaning that \( p(x) = p(y) \) if and only if \( d(x, y) \) is even. We extend the definition of the parity to edges \( xy \in E(Q_n) \). If \( \sum_{i=1}^n x^i < \sum_{i=1}^n y^i \), then \( p(xy) = p(x) \) otherwise, \( p(xy) = p(y) \). A set of pairs \( \{ \{ a_i, b_i \} \}_{i=1}^k \), where \( a_i \) and \( b_i \) are different vertices of \( Q_n \), is considered balanced if the number of white and black vertices are equal. In the below proofs, we generally use that \( Q_n \) is bipartite.

For every $j \in [n]$, consider $Q_{n-1, j}^{0}$ and $Q_{n-1, j}^{1}$ denote two subcubes of $(n-1)$-dimension of a hypercube $Q_n$, where the $j$ superscript be omitted if the context is clear. These subcubes are by vertices of $Q_n$ where the $j$-th coordinate is fixed at 0 or 1, respectively. Note that $Q_n - E_j = Q_{n-1}^1 \cup Q_{n-1}^0$, and we say $Q_n$ is divided into the two subcubes $Q_{n-1}^0$ and $Q_{n-1}^1$ by $E_j$ of $(n-1)$-dimension. For $\theta \in \{0, 1\}$, any vertex $u \in V(Q_{n-1}^\theta)$ has a unique neighbor in $Q_{n-1}^{1-\theta} $, denoted as $u^{1-\theta}$. Similarly, for any edge $uv=e \in E(Q_{n-1}^\theta)$, the corresponding edge in $Q_{n-1}^{1-\theta}$ is $e_{1-\theta}$ denotes the edge  $u_{1-\theta}v_{1-\theta}\in E(Q_{n-1}^{1-\theta})$.
\begin{lemma}\textup{(\cite{T07})}\label{t07}
Consider two disjoint edges $f$ and $e$ in $Q_n $ for $n \geq 2$. Then, split the hypercube $Q_n$ into two $(n-1)$-dimensional subcubes, one has $f$ and the other $e$.
\end{lemma}
A path is a \( u,v \)-path connecting the vertices \( v \) and \( u \) and denoted as \( P_{uv} \). A spanning subgraph of a graph \( G \) consisting of disjoint \( k \)  paths is known as a spanning \( k \)-path of \( G \). In particular, when \( k = 1 \), the spanning 1-path is the Hamiltonian path. If \( e \in E(P) \) then,  path \( P \) is containing the edge \( e \).

Note that the next classical result in \cite{H84} was obtained by Havel.
\begin{theorem}\textup{(\cite{H84})}\label{h84}
Consider two vwertices $x$ and $y$ in $Q_n$ with $p(y)\neq p(x)$. Then, $Q_n$ has a Hamiltonian path between $y$ and $x$.
\end{theorem}
\begin{lemma}\textup{(\cite{DT05})}\label{a2}
Consider $x,y \in V(Q_n)$, where $n\geq 2$, and $f \in E(Q_n)$ with $xy\neq f$ and $p(y)\neq p(x)$, a Hamiltonian path $P_{xy}$ exists in $Q_n$ containing $f$.
\end{lemma}
\begin{theorem}\textup{(\cite{DT05})}\label{dt}
Consider pairwise distinct vertices $u, v, x, y$ of $ Q_n$ where $n \geq 2$, with $p(v) \neq p(u)$ and $p(y) \neq p(x)$. Then, a spanning $2$-path $P_{uv} + P_{xy}$ exists in $Q_n $ furthermore, the path $P_{uv}$ can be chosen such that $P_{uv} = uv$ if $d(u, v) = 1$.
\end{theorem}
\begin{lemma}\textup{(\cite{Dt09})}\label{k2}
Consider \( x, y \in V(Q_n) \)  with \( p(y) \neq p(x) \) and \( f \in E(Q_n) \) where \( n \geq 3 \). Then,  \( Q_n- f \) containing a Hamiltonian path connecting \( y \) and \( x \).
\end{lemma}
\begin{theorem}\textup{(\cite{new14})}\label{TC}
Consider \( n \)-dimensional hypercube \( Q_n \) where \( n \geq 4 \), and let \( F \subseteq E(Q_n) \) with \( |F| \leq 2n - 7 \). For a balanced vertex set \( \{u, v, x, y\} \) in \( Q_n \), then a spanning 2-path \( P_{uv} + P_{xy} \) exists in \( Q_n - F \).
\end{theorem}
\begin{theorem}\textup{(\cite{GP08})}\label{a8}
Consider a balanced pair set $\{(a_i, b_i)\}_{i=1}^k$  in \( Q_n \) where \( n \geq 1 \), with $2k - | \{a_i b_i\}_{i=1}^k \cap E(Q_n)|< n$. Then, a spanning $k$-path \( \sum_{i=1}^{k}P_{a_i b_i} \) exists of \( Q_n \).
\end{theorem}
\begin{lemma}\label{new4}
Consider \( f\) and \( uv \) are two disjoint edges in \( Q_n \) where \( n \geq 5 \). If \( x, y, w, z \) are distinct vertices of $Q_n$ with $ p(w) \neq p(z)$ and $p(x)\neq p(y)$. Then, a spanning 3-path $P_{uv}+P_{xy}+P_{wz}$ exists in \( Q_n - f \).
\end{lemma}
\begin{proof}
By Lemma \ref{t07}, we can split the hypercube \( Q_n \) at direction $j$ into two subcubes \( Q_{n-1}^\theta \) for every $\theta = \{0, 1\}$, such that \( uv \in E(Q_{n-1}^0) \) and \( f \in E(Q_{n-1}^1) \). We may assume \( x \in V(Q_{n-1}^0) \).  
	
If \( y \in V(Q_{n-1}^0) \), let  \( w, z \in V(Q_{n-1}^1) \), since \( p(v) \neq p(u) \) and \( p(y) \neq p(x) \), using Theorem \ref{dt}, a spanning 2-path \( P^0_{xy} + P_{uv}^0 \) exists in \( Q_{n-1}^0 \). So, by Lemma \ref{k2}, a Hamiltonian path \( P^1_{wz} \) exists in \( Q_{n-1}^1-f \). The required spanning 3-path is $P^0_{xy}+ P^0_{uv}+ P^1_{wz}$ in \( Q_n -f \).
	
If \( y \in V(Q_{n-1}^1) \) and $w \in V(Q_{n-1}^0)$, choose $s_0, r_0 \in V(Q_{n-1}^0)$ with \( p(s_0) \neq p(x) \) and \( p(r_0) \neq p(w) \) there exist $s_1, r_1\in V(Q_{n-1}^1)$ with $p(s_1)\neq p(y)$ and $p(r_1)\neq p(z)$. Utilizing Theorem \ref{a8}, a spanning 3-path \( P^0_{xs_0}+ P_{wr_0}^0+ P_{uv}^0 \) exists in \( Q_{n-1}^0 \). Moreover, a spanning 2-path \( P^1_{ys_1} + P^1_{zr_1} \) exists by Theorem \ref{TC} in \( Q_{n-1}^1-f \). Thus, the spanning 3-path in \( Q_n - f\) is $ P_{xy} = P^0_{xs_0}+\{s_0s_1\} + P^1_{ys_1}$, $P_{wz}=P^0_{wr_0}+\{r_0r_1\} + P^1_{zr_1}$, and $P^0_{uv}$.
	
Let consider the case when \( x,y, w,z\in V(Q_{n-1}^1) \). Since \( p(z) \neq p(w) \) and \( p(x) \neq p(y) \). A spanning 2-path \( P^1_{xy}+ P_{wz}^1 \) exists by Theorem \ref{TC}, in \( Q_{n-1}^1-f \), and using Theorem \ref{h84}, a Hamiltonian path \( P^0_{uv}\) exists in \( Q_{n-1}^0\). Thus, the spanning 3-path in \( Q_n - f\) is $P_{xy}^1+P_{wz}^1 +P^0_{uv}$. Assume \( x, y, w, z\in V(Q_{n-1}^0) \), using Theorem \ref{a8}, a spanning 3-path \( P^0_{xy}+ P_{uv}^0+ P_{wz}^0 \) exists in \( Q_{n-1}^0 \). If $P^0_{xy}\geq P^0_{uv}\geq P^0_{wz}$ there exists $s_0r_0 \in E(P_{xy}^0)$ such that $s_1, r_1 \in V(Q_{n-1}^1)$ with $p(r_1)\neq p(s_1)$. Using Lemma \ref{k2}, a Hamiltonian path \( P^1_{s_1r_1}\) exists in \( Q_{n-1}^1-f\). Thus, the required spanning 3-path in \( Q_n - f\) is $ P_{xy}^0 + P^1_{s_1r_1} +\{s_0r_1, s_0r_1\} - s_0r_0$, $P_{wz}^0$ and $P^0_{uv}$.
\end{proof}
Let us consider a vertex $u$, $d_G(u)$ represents the degree of a vertex $u$, and $\delta(G)$ is the minimum degree of $u$ in the graph $G$.
\begin{lemma}\label{new1}
Consider $F$ is a non-empty set of faulty edges in \( Q_n \) where \( n \geq 7 \), such that \( |F| \leq 4n - 17 \). Then in \( Q_{n-1,j}^{\theta} - F_{\theta,j} \), a direction \( j \in \{1, \dots, n\} \) exists with $\delta(Q_{n-1,j}^{\theta} - F_{\theta,j}) \geq 2$, for every \( \theta \in \{0, 1\} \), if \( \delta(Q_n - F) \geq 2 \), one vertex at most having degree 2 in \( Q_n - F \).
\end{lemma}
\begin{proof}
Consider any vertex $u \in V(Q^\theta_{n-1,i})$, where $\theta \in \{0, 1\}$ and $i \in \{1, \dots, n\}$. The degree of $u$ in $Q^\theta_{n-1,i} - F_{\theta, i}$ satisfies the inequality $d_{Q_n - F}(u) - 1 \leq d_{Q^\theta_{n-1,i} - F_{\theta, i}}(u)$.

Case 1. $\delta(Q_n - F) \geq 3$

Here, for each direction $i \in \{1, \dots, n\}$, we see that $\delta (Q_n - F)\geq 2$, there exist $\delta(Q^\theta_{n-1,i} - F_{\theta, i}) \geq 2$ for every $\theta = \{0, 1\}$. Thus, it suffices to find a $j$ direction in $Q^\theta_{n-1,j} - F_{\theta, j}$, at most one vertex has degree 2. Given that $|F| \leq 4n - 17 = 3(n - 3)+ (n - 8)$, the number of vertices with degree 3 in $Q_n - F$ is not greater than 3.  Assume the following three subcases.

Subcase 1.1.  One vertex at most in $Q_n-F$ with degree 3.

Choose arbitrary a $j$ direction where $|F \cap E_j|$ is as large as possible. This ensures that in $Q^\theta_{n-1,j} - F_{\theta, j}$, at most one vertex has degree 2 for each $\theta\in \{0, 1\}$. Thus, $j$ is the required direction. Moreover, $|E_j \cap F| \geq 1$, and also $|E_j \cap F| \geq 2$ if $|F| > n$.

Subcase 1.2. There are exactly two vertices in $Q_n-F$ with degree 3.

Consider in $Q_n - F$, there are two vertices $v$ and $u$ with degree 3. Since, for $n \geq 7$ we have $n < 2(n - 3) $, a $j$ direction exists with $vv_j \in F$ and $uu_j \in F$. This ensures that $\delta(Q^\theta_{n-1,j} - F_{\theta, j}) \geq 3$ for every $\theta=\{0, 1\}$, making $j$ is required direction. Furthermore, if $vv_j \neq uu_j$, then $|E_j \cap F| \geq 2$, if $vv_j = uu_j$ (i.e., $v_j = u$), then $|E_j \cap F| \geq 1$ and $|F_{\theta, j}| \leq (4n-17) -(n-3) = 3n - 14$ for $\theta \in \{0, 1\}$.

Subcase 1.3. There are exactly three vertices in $Q_n-F$ with degree 3.

Since, $|F| \leq 4n - 17 = 3(n - 3)+ (n - 8)$, let $u, v$ and $w$ are the vertices having degree 3 in $Q_n - F$. Since, for $n \geq 7$ we have $3(n - 3) > n$. By symmetry, assume that $vw \in F$ and $uv \in F$. Choose direction $j$ corresponding to $uv$, splitting $Q_n$ at $j$ with $v \in V(Q^1_{n-1})$ and $u \in V(Q^0_{n-1})$ for each $\theta \in \{0, 1\}$, see Figure \ref{Fig.1}. Then $w$ has a degree of at least 2 in $Q^1_{n-1}$ while the remaining vertices have a degree of at least 3 within their respective subcubes. Thus, the chosen direction is $j$. Now, $|E_j \cap F| \geq 1$ and $|F_{\theta, j}| \leq 3n - 14$ foe each $\theta=\{0, 1\}$.
\begin{figure}[H]
	\centering
	\includegraphics[width=0.87\linewidth, height=0.26\textheight]{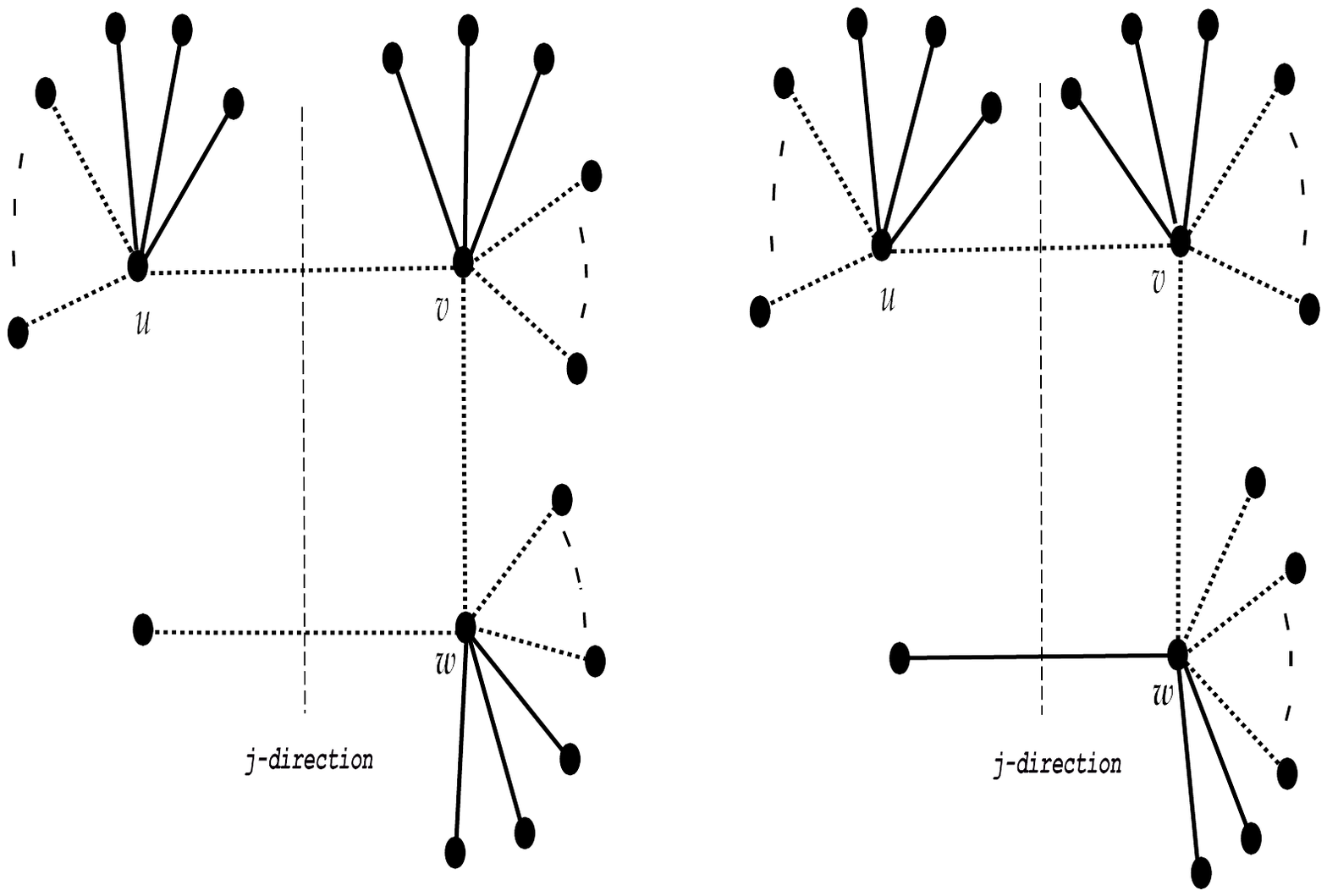}
	\caption[A]{Split $Q_n$ at $j$ direction.}\label{Fig.1}
\end{figure}

Case 2. $\delta(Q_n - F) = 2$.

Since there is exactly a single vertex $w$ in $Q_n - F$ with degree 2. Given $|F| \leq 4n - 17 < 2(n - 3) + (n - 2) + (n-8)$, at most two vertices in $Q_n - F$ has degree 3.

Subcase 2.1. Exactly single vertex in graph $Q_n-F$ with degree 3.

Consider $v$ as the only vertex in $Q_n - F$ containing degree 3. Since, for $n \geq 7$ we have $n + 1 < n - 2 + n - 3$, there exists a $j$ direction with $vv_j \in F$ and $ww_j \in F$, ensuring $vv_j \neq ww_j$. Therefore, $w$ is the only vertex with degree 2 in its respective subcube, while all others have a minimum degree of 3. Thus, $j$ is the required direction. Now, $|E_j \cap F| \geq 2$.

Subcase 2.2. there are exactly two vertices with degree 3 in $Q_n-F$.

Let $u$ and $v$ be the two distinct vertices in $Q_n - F$ has degree 3. Since, for $n \geq 7$ we have $n + 1 < n - 2 + n - 3$, the vertices $v$ and $w$ must be adjacent. There exists a $j$ direction with $wv\in  F$. Therefore, $w$ is the only vertex with degree 2 in its respective faulty subcube, while all others have a minimum degree of 3, shown in Figure \ref{Fig.2}. Thus, $j$ the required direction is the same direction of the line $wv$ with $|E_j \cap F| \geq 1$ and $|F_{\theta, j}| \leq (4n-17)-(n-3) = 3n - 14$ for $\theta \in \{0, 1\}$, there exists at least on vertex at least degree 2 in $Q^\theta_{n-1, i} - F_{\theta, i}$.
\begin{figure}[H]
	\centering
	\includegraphics[width=0.86\linewidth, height=0.25\textheight]{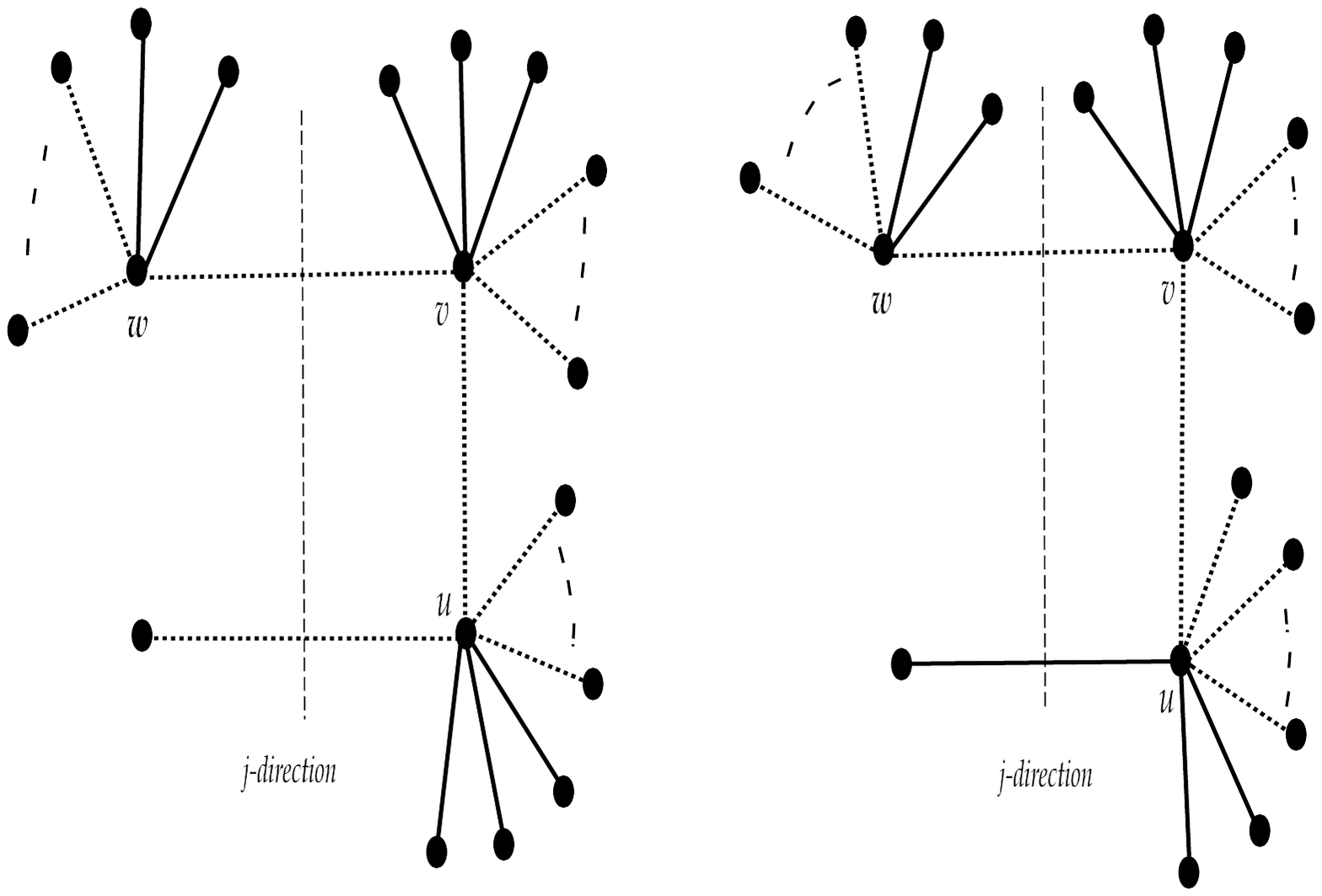}
	\caption[A]{Split $Q_n$ at $j$ direction.}\label{Fig.2}
\end{figure}

Subcase 2.3. In $Q_n-F$ all other vertices with degrees at least 4

Select a $j$ direction where $ww_j \in F$. Then, $w$ is the only vertex that contains degree 2 in their faulty subcube, while all other vertices maintain a degree of at least 3 in their respective subcubes. Confirms $j$ as the suitable direction with $|F \cap E_j| \geq 1$.
\end{proof}
\section{Proof of Theorem \ref{main}}\label{sec3}
If \( F = \emptyset \), the result pursues directly from Theorem \ref{h84}. For the remaining discussion, we consider the case when \( F \neq \emptyset \).  By induction on \( n \), we need to prove the given result. Since \( 4n - 17 \leq 3n - 11 \) for \( n \leq 6 \), by Theorem \ref{n1} result holds. Now, consider the theorem holds for \( n - 1 \) where \( n - 1 \geq 6 \). For \( n \geq 7 \), to show that the result holds. 

Since, \(\delta( Q_n - F)\geq 2 \) and one vertex at most in \( Q_n - F \) having degree 2. There is direction \( j \) according to Lemma \ref{new1}, with \( \delta(Q^{\theta}_{n-1,j} - F_{\theta,j})\geq 2\), and one vertex at most having degree 2 in $Q^{\theta}_{n-1, j} - F_{\theta, j}$ for each $\theta \in \{0, 1\}$. For simplicity, assume $j=c$ define subcubes $Q^{\theta}_{{n-1}, c}$ and faulty edges \( F_{\theta, c} \) for each \( \theta \in \{0, 1\} \), also \( F_c = F \cap E_c \). Using Lemma \ref{new1}, it follows that \( |F_c| \geq 1 \). By symmetry, assume \( |F_0| \geq |F_1| \) which implies \( |F_1| \leq |F_0| \leq 4n - 20 \).

Now, consider \( x \) and \( y \) where \( p(x) \neq p(y) \) are any two vertices in \( Q_n \) . We are to show that a Hamiltonian path \( P_{xy} \) exists in \( Q_n - F \). To proceed, we examine four distinct cases.

Case 1. $|F_0| \leq 4n - 21$. Now, \( |F_1| \leq |F_0| \leq 4(n-1) - 17 \).  

Subcase 1.1. \( x, y \in V(Q_{n-1}^0) \) or \( (x, y \in V(Q_{n-1}^1)) \)

By using induction, a Hamiltonian path \( P^0_{xy} \) exists in \( Q_{n-1}^0 - F_0 \). Consider an edge \( u_0v_0 \) in \( P^0_{xy} \) such that  \( \{u_0u_1, v_0v_1\} \cap F_c =\emptyset \). Since, for $n \geq 7$ we have $|E(P^0_{xy})| - 2|F_c| \geq 2^{n-1} - 1 - 2(4n - 17) > 1$, this is possible all the time. Similarly, using induction hypothesis on \( Q_{n-1}^1 - F_1 \), a Hamiltonian path \( P^1_{u_1 v_1} \) exists. Thus, the required Hamiltonian path in \( Q_n - F \) is $P_{xy} = P^0_{xy} + P^1_{u_1v_1} + \{u_0u_1, v_0v_1\} - u_0v_0$.

Subcase 1.2. \( x \in V(Q_{n-1}^0)\) and \( y \in V(Q_{n-1}^1)  \) or (\( y \in V(Q_{n-1}^0)\) and \( x \in V(Q_{n-1}^1)\)).

Select a vertex \( r_0 \in V(Q_{n-1}^0) \) with \( p(r_0) \neq p(x) \) and \( r_0r_1 \notin F_c \). Since, for $n \geq 7$ we have  $2^{n-2} > |F_c|$, this is possible all the time. Furthermore, it holds that \( p(y) \neq p(r_1) \). Using induction, the Hamiltonian paths \( P^0_{xr_0} \) exist in \( Q_{n-1}^0 - F_0 \) and \( P_{r_1y}^1 \) in \( Q_{n-1}^1 - F_1 \). Thus, the required Hamiltonian path in \( Q_n - F \) is  $P_{xy} = P^0_{xr_0} + P_{r_1y}^1 + r_0r_1$.
	
Case 2. $|F_0| = 4n - 20$.

Given that the edges satisfies \( 1 \leq |F_c| \leq 3 \) and \( |F_1| \leq 2 \). We can observe in the following cases when the number of faulty edges within the subcubes decreases, the degrees of the vertices in those faulty subcubes increase. Therefore, for any subset \( R \subseteq F \), the minimum degree in the graph \( Q_{n-1}^{\theta} - (F_{\theta} \setminus R) \) satisfies $\delta\left(Q_{n-1}^{\theta} - (F_{\theta} \setminus R)\right) \geq 2$. Furthermore, for every  \( \theta \in \{0, 1\} \), there is at most one vertex in \( Q_{n-1}^{\theta} - (F_{\theta} \setminus R) \) having degree exactly 2.

Below, we say that $f$ is an adjacent to $R$, a set of edges if $V(f) \cap V(R) \neq \emptyset$. Moreover, $f$ is incident to $S$, a set of vertices if $V(f) \cap S \neq \emptyset$.

Subcase 2.1. If $y, x \in V(Q_{n-1}^0)$.

Subcase 2.1.1. If an edge $f \in F_0$ exists which is not adjacent to $F_c$.

Since $|F_0| = 4n - 20$ for $n \geq 7$, there must be an edge $f \in F_0$ which is not adjacent to $F_c$. Given that $|F_0 \setminus \{f\}| = 4n - 21 = 4(n - 1) - 17$, applying induction hypothesis, a Hamiltonian path $P^0_{xy}$ exists in $Q_{n-1}^0 - (F_0 \setminus \{f\})$. We can select an edge from the path $P^0_{xy}$, based on the rules listed below

(i) If $P^0_{xy}$ contains $f$, then $f$ is chosen.

(ii) Otherwise, we select an edge not adjacent to $F_c$ from $P^0_{xy}$. Since $|E(P^0_{xy})| = 2^{n-1} - 1 > 2|F_c|$, this is possible all the time. 
	
Let $u_0v_0$ denote the chosen edge. Now, $\{u_0u_1, v_0v_1\} \cap F_c = \emptyset$. Since $|F_1| \leq 2$, using Theorem \ref{n1}, a Hamiltonian path $P^1_{u_1v_1}$ exists in $Q_{n-1}^1 - F_1$. Thus, the required Hamiltonian path in $Q_n - F$ is  $P_{xy} = P^0_{xy} + P^1_{u_1v_1} + \{u_0u_1, v_0v_1\} - u_0v_0$.

Subcase 2.1.2.  All edges in $F_0$ adjacent to $F_c$.

Since $|F_0| = 4n - 20 <  (n - 4) + (n - 3) + (n - 4)$ for $n \geq 7$, therefore all edges of $F_0$ are adjacent to $F_c$. Given that $|F_0 \setminus \{f\}| = 4n - 21 = 4(n - 1) - 17$, applying induction hypothesis, a Hamiltonian path $P^0_{xy}$ exists in $Q_{n-1}^0 - (F_0 \setminus \{f\})$. 

Let $u_0v_0$ denote the chosen edge. Now, $\{u_0u_1, v_0v_1\} \cap F_c \neq \emptyset$, let $u_0u_1 \in F_c$. There exists a neighbor $u'_0$ of $u_0$ with $u_0u'_0 \notin F_0$, that is adjacent to $v'_0$, on other side of $P^0_{xy}$ such that $p(v_1)\neq p(v'_1)$. Since $F_1=\emptyset$, by Theorem \ref{h84}, a Hamiltonian path $P^1_{v_1v'_1}$ exists in $Q_{n-1}^1$. Hence, the required Hamiltonian path in $Q_n - F$ is  $P_{xy} = P^0_{xy} + P^1_{v_1v'_1} + \{v_0v_1, v'_0v'_1, u_0u'_0\} - \{u_0v_0, u'_0v'_0\}$.

Subcase 2.2. $y \in V(Q_{n-1}^1)$ and $x \in V(Q_{n-1}^0)$.
  
Subcase 2.2.1. The edge $u_0v_0 \in F_0$ exists not adjacent to $F_c$, i.e., $\{u_0u_1, v_0v_1\} \cap F_c  = \emptyset$.

Choose $r_0 \in V(Q_{n-1}^0) \setminus \{u_0, v_0\}$ with $p(r_0) \neq p(x)$ and $r_0r_1 \notin F_c$. Since $2^{n-2} - 1 > |F_c|$, such an $r_0$ always exists. Furthermore, $p(y) \neq p(r_1)$, and the vertices $r_1, y, u_1, v_1$ are distinct. Since $|F_0 \setminus \{u_0v_0\}| = 4(n - 1) - 17$ by induction a Hamiltonian path $P^0_{xr_0}$ exists in $Q_{n-1}^0 - (F_0 \setminus \{u_0v_0\})$. So, $|F_1| \leq 2$, if $u_0v_0 \notin E(P^0_{xr_0})$, using Theorem \ref{n1}, a Hamiltonian path $P^1_{r_1y}$ exists in $Q_{n-1}^1 - F_1$. Thus,  the required Hamiltonian path is $P_{xy} = P^0_{xr_0} + r_0r_1 + P^1_{r_1y}$ in $Q_n-F$. Otherwise, if $u_0v_0 \in E(P^0_{xr_0})$, a spanning 2-path $P^1_{u_1v_1} + P^1_{r_1y}$ exists by Theorem \ref{TC} in $Q_{n-1}^1 - F_1$. The required Hamiltonian path is $P_{xy} = P^0_{xr_0} + P^1_{u_1v_1} + P^1_{r_1y} + \{u_0u_1, v_0v_1, r_0r_1\} - u_0v_0$, see Figure \ref{Fig.3}(a).

Subcase 2.2.2. There exists $u_0v_0 \in F_0$ adjacent to $F_c$, i.e., $\{u_0u_1, v_0v_1\} \cap F_c  \neq \emptyset$.

If $|F_c| = 3$ implying $F_1 = \emptyset$. Since $|F_0| = 4n - 20 < (n - 4) + (n - 3) +  (n - 4)$, all edges in $F_0$ are adjacent to $F_c$. Let $u_0v_0\in F_0$ and choose $r_0\in V(Q_{n-1}^0)$ such that $p(x)\neq p(r_0)$. By induction a Hamiltonian path $P^0_{xr_0}$ exists in $Q_{n-1}^0 - (F_0 \setminus \{u_0v_0\})$. If $u_0v_0 \notin E(P^0_{xr_0})$, using Theorem \ref{h84}, a Hamiltonian path $P^1_{r_1y}$ exists in $Q_{n-1}^1$, the Hamiltonian path is $P_{xy} = P^0_{xr_0} + r_0r_1 + P^1_{r_1y}$. Otherwise, if $u_0v_0 \in E(P^0_{xr_0})$, let $u_0u_1\in F_c$, we select a neighbor $u'_0$ of $u_0$ with $u_0u'_0 \notin F_0$ and $v'_0$ is adjacent to  $u'_0$ on other side of $P^0_{xr_0}$ such that $p(v_0)\neq p(v'_0)$. Since $r_1, y, v_1, v'_1$ are distinct vertices, using Theorem \ref{dt}, a spanning 2-path $P^1_{v_1v'_1} + P^1_{yr_1}$ exists in $Q_{n-1}^1$. Hence, the required Hamiltonian path in $Q_n-F$ is $P_{xy} = P^0_{xr_0} + P^1_{v_1v'_1}+ P^1_{yr_1} + \{v_0v_1, v'_0v'_1, r_0r_1, u_0u'_0\} - \{u_0v_0, u'_0v'_0\}$, as shown in Figure \ref{Fig.3}(b). 
\begin{figure}[H]
	\centering
	\includegraphics[width=0.85\linewidth, height=0.26\textheight]{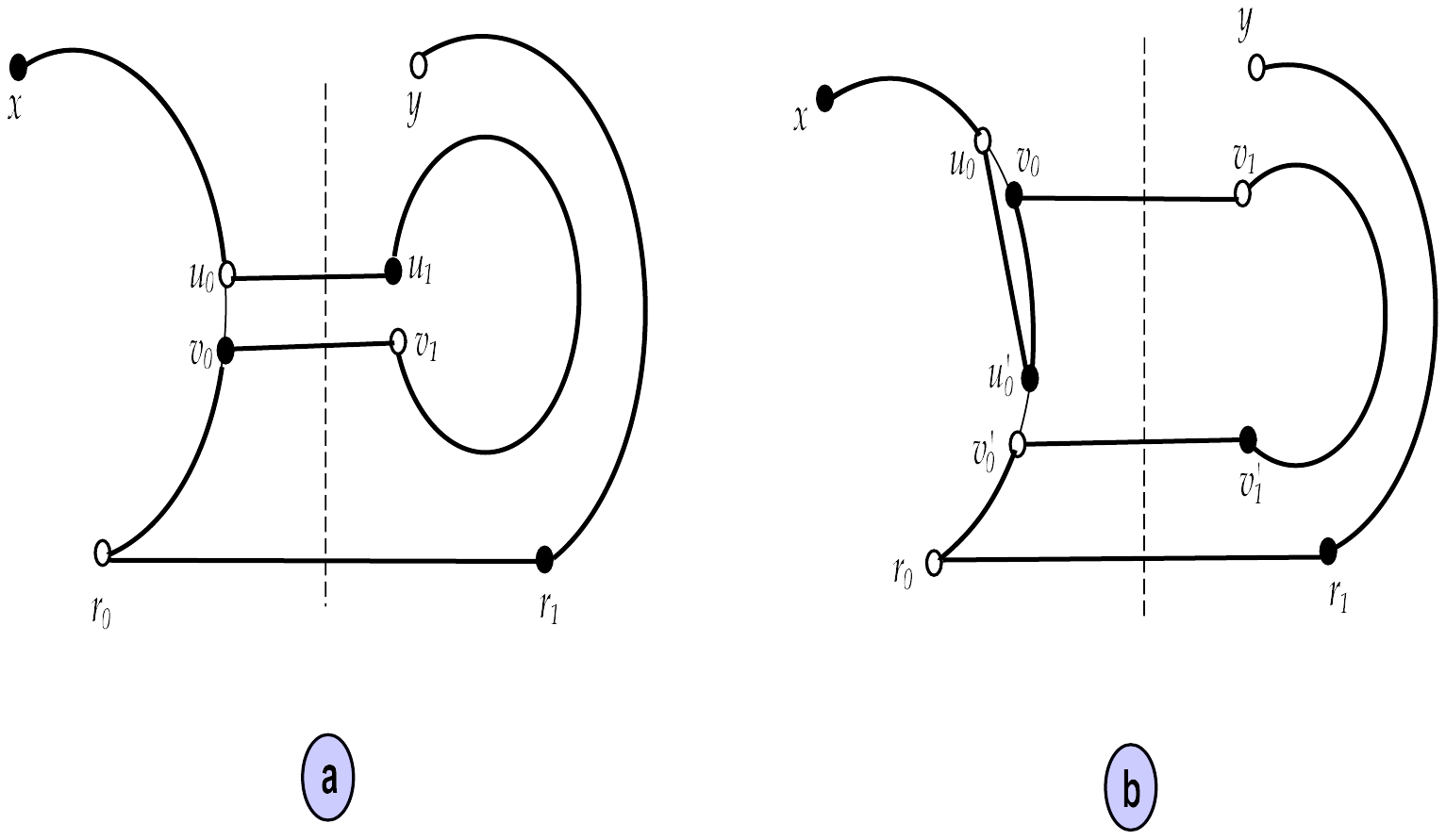}
	\caption[A]{Illustration of Subcase 2.2.}\label{Fig.3}
\end{figure}

Subcase 2.3. \( x, y \in V(Q_{n-1}^0) \).  

Given that \( |F_0| = 4n - 20  \), choose an edge $u_0v_0\in F_0$ such that $\{u_1, v_1\}\in V(Q_{n-1}^1\setminus\{x, y\})$ and $\{u_0u_1, v_0v_1\}\cap F_c = \emptyset$. Since $|F_0 \setminus \{u_0v_0\}| = 4(n - 1) - 17$, using induction hypothesis, a Hamiltonian path \( P_{u_0v_0}^0 \) exists in \( Q_{n-1}^0 - (F_0\setminus u_0v_0) \). Since \( p(u_1)\neq p(x) \) and $ p(v_1)\neq p(y)$ with \( |F_1| \leq 2 \), by Theorem \ref{TC} or Theorem \ref{dt} (when $F_1=\emptyset$), a spanning 2-path \( P^1_{xu_1} + P^1_{yv_1}\) exists in \( Q_{n-1}^1 - F_1 \). Thus, the required Hamiltonian path in \( Q_n - F \) is  $P_{xy} = P^0_{u_0v_0} + \{u_0u_1, v_0v_1\}+ P^1_{xu_1}+  P^1_{yv_1}$. 

If $|F_c| = 3$ implying $F_1 = \emptyset$, choose an edge $u_0v_0\in F_0$ such that $\{u_1, v_1\}\in V(Q_{n-1}^1\setminus\{x, y\})$ with $\{u_0u_1, v_0v_1\}\cap F_c \neq \emptyset$, let $u_0u_1\in F_c$, we choose a neighbor $u'_0$ of $u_0$ with $u_0u'_0 \notin F_0$ and $v'_0$ is adjacent to  $u'_0$ on other side of $P^0_{u_0v_0}$ such that $p(v_0)\neq p(v'_0)$. Since $x, y, v_1, v'_1$ are distinct vertices, with $p(x)\neq p(v'_1)$ and $p(y)\neq p(v_1)$, a spanning 2-path $P^1_{xv'_1}+ P^1_{yv_1}$ exists by Theorem \ref{dt} in $Q_{n-1}^1-F_1$. Hence, the required Hamiltonian path in $Q_n-F$ is $P_{xy} = P^0_{u_0v_0} + P^1_{xv'_1}+ P^1_{yv_1} + \{v_0v_1, v'_0v'_1, u_0u'_0\} -  u'_0v'_0$.

Case 3: $|F_0| = 4n - 19$.

Given that the faulty edges satisfies \( 1 \leq |F_c| \leq 2 \) and \( |F_1| \leq 1 \). We can observe in the subsequent cases when the number of faulty edges within the subcubes decreases, the degrees of the vertices in those faulty subcubes increase. Thus, for any subset \( R \subseteq F \), the minimum degree in the graph \( Q_{n-1}^{\theta} - (F_{\theta} \setminus R) \) satisfies $\delta\left(Q_{n-1}^{\theta} - (F_{\theta} \setminus R)\right) \geq 2$. Furthermore, for each \( \theta \in \{0, 1\} \), there is at most one vertex in \( Q_{n-1}^{\theta} - (F_{\theta} \setminus R) \) with degree exactly 2.

We define $f$ and $e$ as adjacent edges to $R$, a set of edges if $V(\{f, e\}) \cap V(R) \neq \emptyset$. Moreover, edges $f$ and $e$ are incident to a set of vertices $S$, if $V(e) \cap S \neq \emptyset$ and $V(f) \cap S \neq \emptyset$.

Subcase 3.1. If $y, x \in V(Q_{n-1}^0)$.

Subcase 3.1.1. Edges of $F_0$ that are not adjacent to $F_c$.

Since $|F_0| = 4n - 19$ for $n \geq 7$, the disjoint edges $f, e \in F_0$ exists that is not adjacent to $F_c$. Given that $|F_0 \setminus \{f, e\}| = 4n - 21 = 4(n - 1) - 17$, applying induction hypothesis, a Hamiltonian path $P^0_{xy}$ exists in $Q_0 - (F_0 \setminus \{f, e\})$. We can select edges from the path $P^0_{xy}$ based on the rules listed below.

(i) If $P^0_{xy}$ containing $f$ and $e$, then $f, e$ are chosen.

(ii) If $P_{xy}^0$ path passes through only one edge of edges $f$ and $e$, then $f$ is chosen. Next, we choose an edge of $P_{xy}^0$ that is neither adjacent to $f$ nor incident with $F_c$. Since $|E(P^0_{xy})| = 2^{n-1} - 1 > 3 + 2$, this is always all the time.

(iii) Otherwise, we select edges from $P^0_{xy}$ that are not adjacent to $F_c$. Since $|E(P^0_{xy})| > 4$, this is possible all the time. 

Let $u_0v_0$ and $w_0z_0$ denote the chosen edges. Now, $\{u_0u_1, v_0v_1, w_0w_1, z_0z_1\} \cap F_c = \emptyset$ and $|F_1| \leq 1$. Since $n-1\geq 6$, a spanning 2-path $P^1_{u_1v_1} + P^1_{w_1z_1}$ exists by Theorem \ref{TC} in $Q_{n-1}^1 - F_1$. Thus, the required Hamiltonian path in $Q_n - F$ is $P_{xy} = P^0_{xy} + P^1_{u_1v_1}  + P^1_{w_1z_1} + \{u_0u_1, v_0v_1, w_0w_1, z_0z_1\} - \{u_0v_0, w_0z_0\}$, see in Figure \ref{Fig.7}(a).

Subcase 3.1.2. Edges of $F_0$ that are adjacent to $F_c$.

If $|F_c| = 2$ implying $F_1 = \emptyset$, since for $n \geq 7$, we have $|F_0| = 4n - 19 > (n - 4) + (n - 3)$, there must be an edge not adjacent to $F_c$. Let $f \in F_0$ not adjacent to $F_c$ and $e \in F_0$ that is adjacent to $F_c$. Given that $|F_0 \setminus \{f, e\}| = 4n - 21 = 4(n - 1) - 17$, applying induction hypothesis, a Hamiltonian path $P^0_{xy}$ exists in $Q_0 - (F_0 \setminus \{f, e\})$. We can select edges from the path $P^0_{xy}$ based on the rules listed below.

(i) If $P^0_{xy}$ containing $f$ and $e$, then $f, e$ are chosen.

(ii) If $P_{xy}^0$ path passes through only one edge of edges $f$ and $e$, then $f$ is chosen. Next, we choose an edge of $P_{xy}^0$ that is neither adjacent to $f$ nor incident with $F_c$. Since $|E(P^0_{xy})| = 2^{n-1} - 1 > 3 + 2$, this is possible all the time.

(iii) Otherwise, we select edges from $P^0_{xy}$ that are not adjacent to $F_c$. Since $|E(P^0_{xy})| > 4$, this is possible all the time. 

Let $u_0v_0$ and $w_0z_0$ denote the chosen edges. Now, $\{u_0u_1, v_0v_1, w_0w_1, z_0z_1\} \cap F_c \neq \emptyset$, let $u_0u_1 \in F_c$. We select a neighbor $u'_0$ of $u_0$ with $u_0u'_0 \notin F_0$, and $v'_0$ is adjacent to  $u'_0$ on other side $P^0_{xy}$ such that $p(v_0)\neq p(v'_0)$. Since $w_1, z_1, v_1, v'_1$ are distinct vertices, using Theorem \ref{dt}, a spanning 2-path $P^1_{v_1 v'_1}+ P^1_{w_1z_1}$ exists in $Q_{n-1}^1$. Thus, the required Hamiltonian path in $Q_n - F$ is $P_{xy} = P^0_{xy} + P^1_{v_1 v'_1}+ P^1_{w_1z_1} + \{v_0v_1, v'_0v'_1, w_0w_1, z_0z_1, u_0u'_0\} - \{u_0v_0, u'_0v'_0, w_0z_0\}$, shown in Figure \ref{Fig.7}(b).
\begin{figure}[H]
	\centering
	\includegraphics[width=0.9\linewidth, height=0.26\textheight]{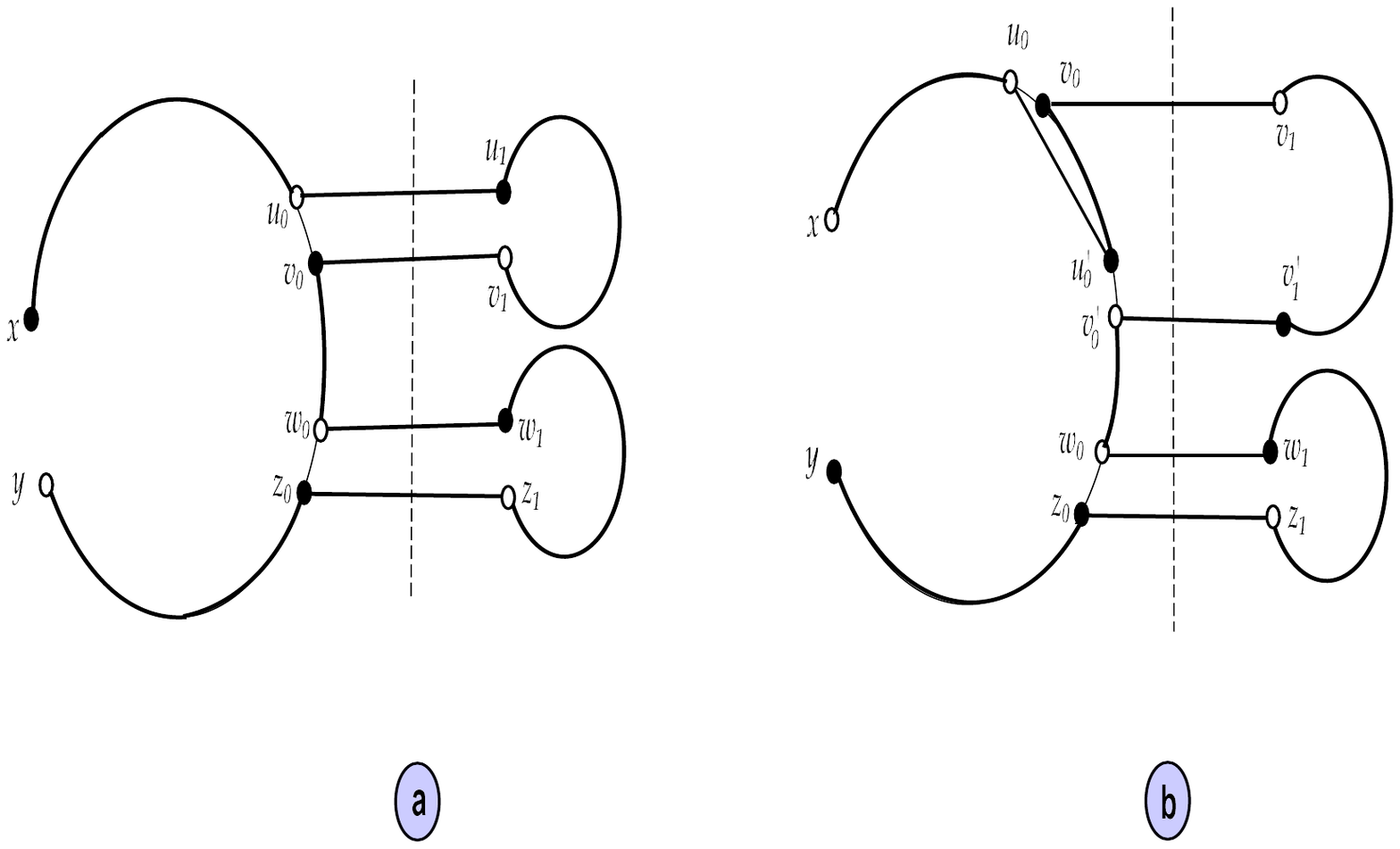}
	\caption[A]{Depiction of Subcase 3.1.}\label{Fig.7}
\end{figure}

Subcase 3.2. $y \in V(Q_{n-1}^1)$ and $x \in V(Q_{n-1}^0)$.

Subcase 3.2.1: There exist disjoint edges $\{u_0v_0, w_0z_0\} \in F_0$ that are not adjacent to $F_c$.

Choose $r_0 \in V(Q_{n-1}^0) \setminus \{u_0, v_0, w_0, z_0\}$ with $p(x) \neq p(r_0)$ and $r_0r_1 \notin F_c$. Since $2^{n-2} - 1 > |F_c|$, such an $r_0$ always exists. Furthermore, $p(y) \neq p(r_1)$, and the vertices $u_1, v_1, w_1, z_1, r_1, y$ are distinct. Since $|F_0 \setminus \{u_0v_0, w_0z_0\}| = 4(n - 1) - 17$, by induction a Hamiltonian path $P^0_{xr_0}$  exists in $Q_{n-1}^0 - (F_0 \setminus \{u_0v_0,  w_0z_0\})$. Since $|F_1| \leq 1$, if $\{u_0v_0,  w_0z_0\} \notin E(P^0_{xr_0})$, then according to Theorem \ref{n1}, a Hamiltonian path $P^1_{r_1y}$ exists in $Q_{n-1}^1 - F_1$. Thus,  $P_{xy} = P^0_{xr_0} + r_0r_1 + P^1_{r_1y}$ is required Hamiltonian path in $Q_n-F$. If one of these edges are not in $E(P^0_{xr_0})$ i.e., $\{u_0v_0\} \notin E(P^0_{xr_0})$, for $n-1\geq 6$, then using Theorem \ref{TC}, a spanning 2-path $P^1_{r_1y} + P^1_{w_1z_1}$ exists in $Q_{n-1}^1 - F_1$. Thus,  $P_{xy} = P^0_{xr_0} + P^1_{r_1y}+ P^1_{w_1z_1}+ \{r_0r_1, w_0w_1, z_0z_1\} - w_0z_0$ is required Hamiltonian path in $Q_n-F$. Otherwise, if $\{u_0v_0,  w_0z_0\} \in E(P^0_{xr_0})$, a spanning 3-path $P^1_{u_1v_1} + P^1_{w_1z_1}+ P^1_{r_1y}$ exists by Lemma \ref{new4} in $Q_{n-1}^1 - F_1$. The required Hamiltonian path is $P_{xy} = P^0_{xr_0} + P^1_{u_1v_1} + P^1_{w_1z_1}+ P^1_{r_1y} + \{u_0u_1, w_0w_1, z_0z_1, v_0v_1, r_0r_1\} - \{u_0v_0, w_0z_0\}$, see Figure \ref{Fig.4}(a). 

Subcase 3.2.2. Except one all edges in $F_0$ are adjacent to $F_1$.

Since $|F_c| = 2$ implying $F_1 = \emptyset$. As we know $|F_0| = 4n - 19 >  (n - 4) + (n - 3) $, there must be an edge not adjacent to $F_c$. Let $w_0z_0 \in F_0$ not adjacent to $F_c$, and $u_0v_0 \in F_0$ that is adjacent to $F_c$. Let $\{u_0v_0, w_0z_0\}\in F_0$, choose $r_0 \in V(Q_{n-1}^0) \setminus \{u_0, v_0, w_0, z_0\}$ such that $p(r_0) \neq p(x)$ and $r_0r_1 \notin F_c$. Using induction hypothesis, a Hamiltonian path $P^0_{xr_0}$ exists in $Q_{n-1}^0 - (F_0 \setminus \{u_0v_0, w_0z_0\})$. If $\{u_0v_0, w_0z_0\} \notin E(P^0_{xr_0})$, then by Theorem \ref{h84}, a Hamiltonian path $P^1_{r_1y}$  exists in $Q_{n-1}^1$. Thus, the Hamiltonian path is $P_{xy} = P^0_{xr_0} + r_0r_1 + P^1_{r_1y}$. If one of these edges are not in $E(P^0_{xr_0})$ i.e, $\{u_0v_0\} \notin E(P^0_{xr_0})$, then a spanning 2-path $P^1_{r_1y} + P^1_{w_1z_1}$ exists by Theorem \ref{dt} in $Q_{n-1}^1$. Thus,  $P_{xy} = P^0_{xr_0}+ P^1_{r_1y}+ P^1_{w_1z_1}+ \{r_0r_1, w_0w_1, z_0z_1\} - w_0z_0$ is required Hamiltonian path in $Q_n-F$. Otherwise, if $\{u_0v_0, w_0z_0\} \in E(P^0_{xr_0})$ and $u_0u_1\in F_c$, we select a neighbor $u'_0$ of $u_0$ with $u_0u'_0 \notin F_0$, and $v'_0$ is adjacent to  $u'_0$ on other side of $P^0_{xr_0}$, such that $p(v_1)\neq p(v'_1)$. Since $r_1, y, v_1, v'_1, w_1, z_1$ are distinct vertices and $\{r_1, y, v_1, v'_1, w_1, z_1\}$ are balanced, by Theorem \ref{a8}, a spanning 3-path $P^1_{v_1v'_1}+ P^1_{yr_1} + P^1_{w_1z_1}$ exists in $Q_{n-1}^1$. Thus, the required Hamiltonian path in $Q_n-F$ is $P_{xy} = P^0_{xr_0} + P^1_{v_1v'_1}+ P^1_{yr_1} + P^1_{w_1z_1} + \{w_0w_1, z_0z_1, v_0v_1, v'_0v'_1, r_0r_1, u_0u'_0\} - \{u_0v_0, u'_0v'_0, w_0z_0\}$, see Figure \ref{Fig.4}(b).
\begin{figure}[H]
	\centering
	\includegraphics[width=0.9\linewidth, height=0.26\textheight]{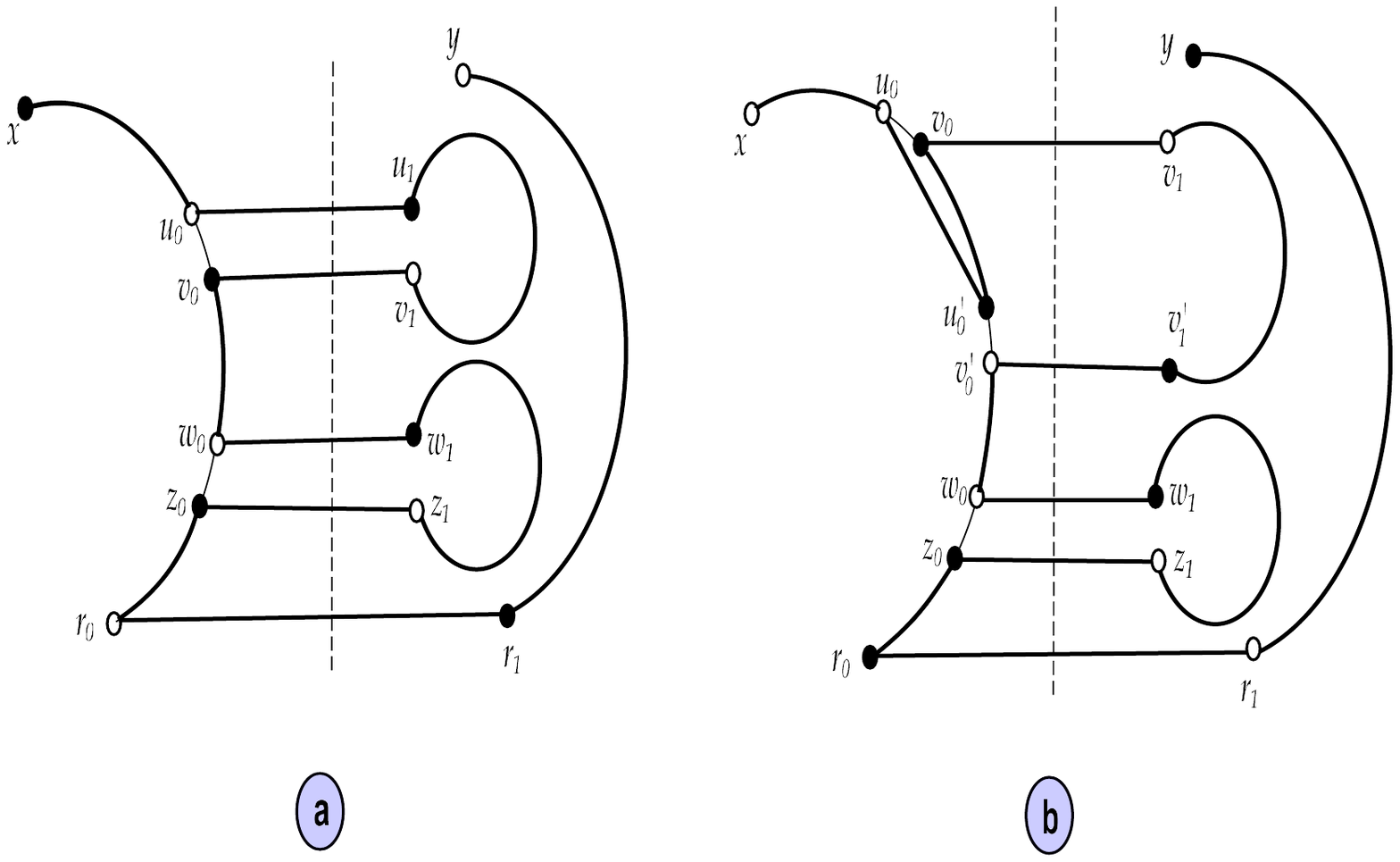}
	\caption[A]{Depiction of Subcase 3.2.}\label{Fig.4}
\end{figure}

Subcase 3.3. \( x, y \in V(Q_{n-1}^1) \).  

Given that \( |F_0| = 4n - 19  \), choose disjoint edges \{$u_0v_0, w_0z_0\}\in F_0$ such that $\{u_1, v_1, w_1, z_1\}\in V(Q_{n-1}^1\setminus\{x, y\})$ and $\{u_0u_1, v_0v_1, w_0w_1, z_0z_1\}\cap F_c = \emptyset$. Since $|F_0 \setminus \{u_0v_0, w_0z_0\}| = 4(n - 1) - 17$, by induction there exists a Hamiltonian path \( P_{u_0v_0}^0 \) in \( Q_{n-1}^0 - (F_0\setminus \{u_0v_0, w_0z_0\}) \). If $w_0z_0 \in  E(P_{u_0v_0}^0) $, such that \( p(u_1)\neq p(x) \) and $ p(v_1)\neq p(y)$. Since \( |F_1| \leq 1 \), by Lemma \ref{new4} or Theorem \ref{a8} (when $F_1=\emptyset$) a spanning 3-path \( P^1_{xu_1} + P^1_{yv_1} + P^1_{w_1z_1}\) exists in \( Q_{n-1}^1 - F_1 \). Thus, the required Hamiltonian path in \( Q_n - F \) is  $P_{xy} = P^0_{u_0v_0} + P^1_{xu_1} + P^1_{yv_1} + P^1_{w_1z_1} + \{u_0u_1, v_0v_1, w_0w_1, z_0z_1\}- w_0z_0 $, shown in Figure \ref{Fig.5}(a). If $w_0z_0 \notin  E(P_{u_0v_0}^0)$, we proceed as in Subcase 2.3.

If $|F_c| = 2$ implying $F_1 = \emptyset$. Since $|F_0| = 4n - 19 >  (n - 4) + (n - 3) $ for $n \geq 7$, there must be an edge not adjacent to $F_c$. If $w_0z_0 \in  E(P_{u_0v_0}^0)$, let $w_0z_0 \in F_0$ that is adjacent to $F_c$. So, $\{w_0w_1, z_0z_1\}\cap F_c \neq \emptyset$, let $w_0w_1\in F_c$, we select a neighbor $w'_0$ of $w_0$ with $w_0w'_0\notin F_0$, and $z'_0$ is adjacent to  $w'_0$ on other side of $P^0_{u_0v_0}$, such that $p(z_1)\neq p(z'_1)$. Since $x, y, z_1, z'_1, u_1, v_1$ are distinct vertices with $p(x)\neq p(u_1)$ and $p(y)\neq p(v_1)$, using Theorem \ref{a8}, a spanning 3-path $P^1_{xu_1}+ P^1_{yv_1} + P^1_{z_1z'_1}$ exists in $Q_{n-1}^1$. Thus, the required Hamiltonian path in $Q_n-F$ is $P_{xy} = P^0_{u_0v_0} + P^1_{xu_1}+ P^1_{yv_1} + P^1_{z_1z'_1} + \{u_0u_1, v_0v_1, z_0z_1, z'_0z'_1, w_0w'_0\} -  \{w_0z_0, w'_0z'_0\}$, see in Figure \ref{Fig.5}(b). If $w_0z_0 \notin  E(P_{u_0v_0}^0)$,  we proceed as in Subcase 2.3.
\begin{figure}[H]
	\centering
	\includegraphics[width=0.9\linewidth, height=0.26\textheight]{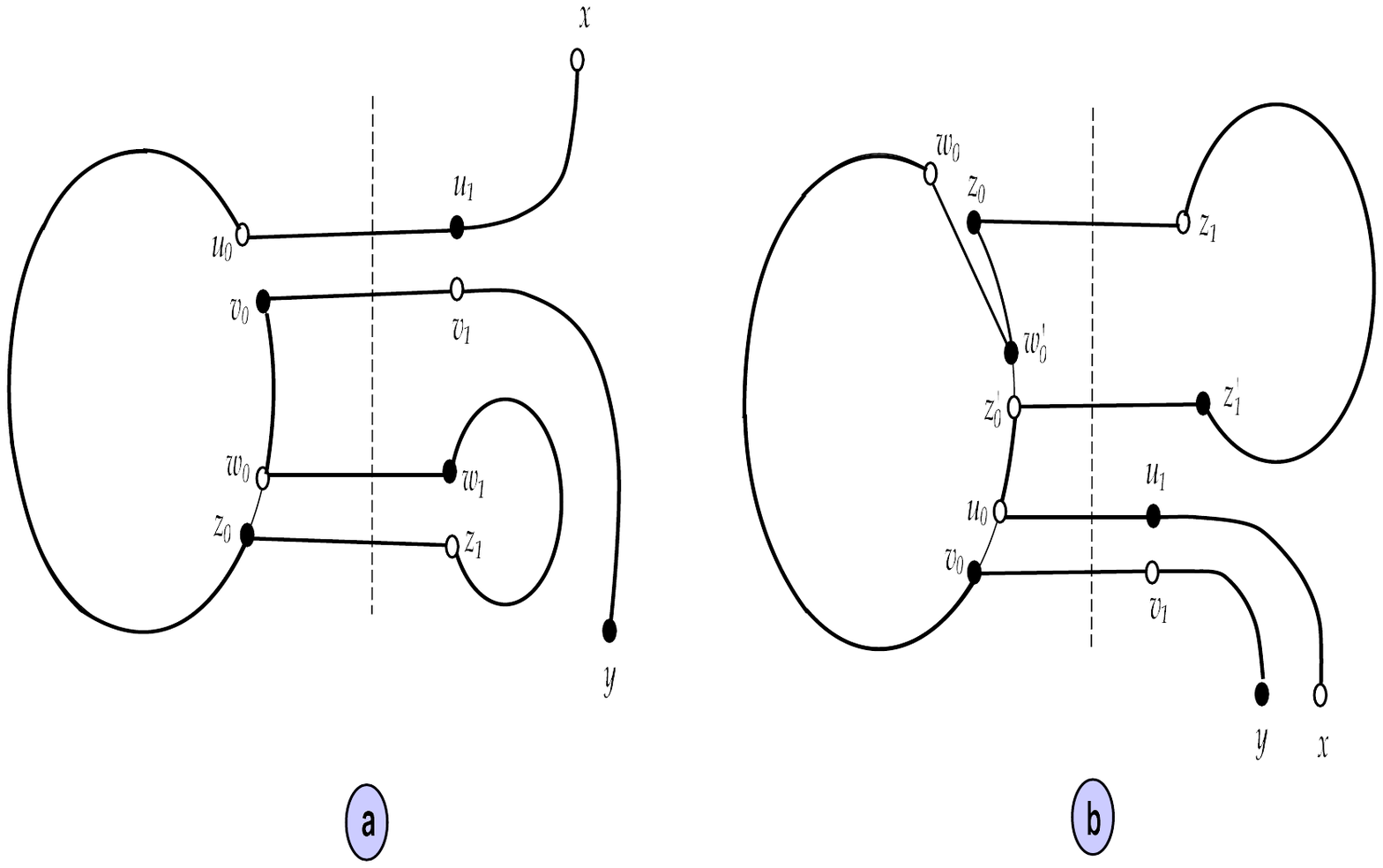}
	\caption[A]{Depiction of Subcase 3.3.}\label{Fig.5}
\end{figure}

Case 4: $|F_0| = 4n - 18$.

Now we have $|F| = 4n - 17>n$, with $|F_0| = 4n - 18 > 3n - 14$. Moreover, $|F_c| = 1$, where $F_c$ consists of a single edge ${t_0t_1}$ with $t_0\in V(Q_{n-1}^0)$. This case can only arise in the proof of Lemma \ref{new1} in Subcase 2.3. Thus, we can consider, for any vertex $v\in V(Q_n)-t_0$ that is $d_{Q_n-F}(t_0) = 2$ and $d_{Q_n-F}(v) \geq 4$.

First, consider a vertex $u$ exists fulfilling $d_{Q_n-F}(u) = 4$. For $n \geq 7$, since $(n - 2) + (n - 4) > n$ there must be a $k$ direction  where $uu_k\in F$ and $t_0t_k\in F$. Divide $Q_n$ at $k$ direction, the faulty subcube containing $t_0$ has exactly one vertex having degree $2$, while all other vertices in their respective faulty subcubes have a degree of at least $3$. Thus, $\delta (Q_{n-1,k}^{\theta} - F_{\theta,k})\leq 2$, mostly one vertex having degree $2$ in $Q_{n-1,k}^{\theta} - F_{\theta,k}$ for each $\theta \in \{0, 1\}$. If $t_0t_k \neq uu_k$, then $|F\cap E_k|=2$, with $|F_{\theta,k}| \leq 4n - 19$ for both $\theta \in \{0, 1\}$. We can reduce this situation to Cases 1, 2, or Case 3. If $uu_k = t_0t_k$ (i.e., $w_k = u$), then $|E_k \cap F| \geq 1$ and $|F_{\theta,k}| \leq 4n - 17 - (n - 4) \leq 4n - 21$ for all $\theta \in \{0, 1\}$, This reduces to the same situation as Case 1. So the below we  may consider for any vertex $v_0 \in V(Q_n)-t_0$ that is $d_{Q_n-F}(v) \geq 5$.

Subcase 4.1. $x, y\in V(Q_{n-1}^0)$.

Choose three disjoint edges $f, e, g\in F_0$ such that $t_0\notin V(f)\cup V(e)\cup V(g)$. For $n \geq 7$, since $|F_0| = 4n - 18> (n - 5) + (n - 3) + (n - 5)$ this is possible all the time. Since $|F_0 \setminus \{f, e, g\}| = 4(n - 1) - 17$, and by induction a Hamiltonian path $P^0_{xy}$ exists in $Q_{n-1}^0 - (F_0 \setminus \{f, e, g\})$. We select three disjoint edges from $P^0_{xy}$ according to the below rules.

(i) If $P^0_{xy}$ contains both $f, e$ and $g$, these edges can be chosen.

(ii) If the path $P_{xy}^0$ containing only one of edges $f, e$ and $g$, then $f$ is chosen. Next, we choose edges of $P_{xy}^0$ that are neither adjacent to $f$ nor incident with $t_0$. Since $|E(P^0_{xy})| = 2^{n-1} - 1 > 6$, this is possible all the time.

(iii) If the path $P_{xy}^0$ containing two edges of $f, e$ and $g$, then $f, e$ are chosen. Next, we choose an edge of $P_{xy}^0$ that is neither adjacent to $f$ and $e$ nor incident with $t_0$. Since $|E(P^0_{xy})| = 2^{n-1} - 1 > 5$, this is possible all the time.

(iv) Otherwise, choose three disjoint edges from $P^0_{xy}$ are chosen, ensuring that are not incident with $t_0$. Since $|E(P^0_{xy})| > 5$, this is possible all the time.

Suppose these edges denoted by $u_0v_0, w_0z_0$ and $s_0r_0$, we ensure $u_0, v_0, w_0, z_0, s_0$, and $r_0$ are distinct, and $t_0t_1 \notin \{u_0u_1, v_0v_1,w_0w_1, z_0z_1, s_0s_1, r_0r_1\}$.  The vertices $\{u_0, v_0, w_0, z_0, s_0, r_0\}$ are balanced, by Theorem \ref{a8}, a spanning 3-path $P^1_{u_1v_1} + P^1_{w_1z_1} + P^1_{s_1r_1}$ exists in $Q_{n-1}^1$. Hence, the Hamiltonian path $P_{xy}$ in $Q_n - F$ is $P_{xy} =P^0_{xy} + P^1_{u_1v_1} + P^1_{w_1z_1} + P^1_{s_1r_1}+ \{u_0u_1, w_0w_1, z_0z_1, v_0v_1, s_0s_1, r_0r_1\}-\{u_0v_0, w_0z_0, s_0r_0\}$.

Subcase 4.2. $y \in V(Q_{n-1}^1)$ and $x \in V(Q_{n-1}^0)$.

Since $|F_0| = 4n - 18 > (n - 5) + (n - 5) + (n - 3)$ for $n\geq7$, three disjoint edges $f, e$ and $g$ exist in $F_0$ such that $V(\{f, e, g\})\cap \{t_0, y_0\}=\emptyset$. Choose a vertex $m_0$ in $Q_{n-1}^0$ satisfying $p(m_0) \neq p(x)$ and $m_0 \notin \{t_0\} \cup V(\{f, e, g\})$, such an $m_0$ always exists as $2^{n-2} > 6$. Moreover, $p(m_1) \neq p(y)$, and $m_0m_1 \neq t_0t_1$. Since $|F_0 \setminus \{f, g, g\}| = 4(n - 1) - 17$, using induction, a Hamiltonian path $P^0_{xm_0}$ exists in $Q_{n-1}^0 - (F_0 \setminus \{f, e, g\})$. We select three disjoint edges from $P^0_{xy}$ according to the rules listed below.
 
(i) If $P^0_{xy}$ contains both $f, e$ and $g$, these edges can be chosen.
 
(ii) If the path $P_{xy}^0$ contains only one of edges $f, e$ and $g$, then $f$ is chosen. Next, we select edges of $P_{xy}^0$ that is non-adjacent to ${m_0, t_0, y_0}$. Since $|E(P^0_{xy})| = 2^{n-1} - 1 > 5 + 4$, this is possible always.
 
(iii) If the path $P_{xy}^0$ containing two edges of $f, e$ and $g$, then $f, e$ is chosen. Next, we choose edges of $P_{xy}^0$ that is non-adjacent to ${m_0, t_0, y_0}$. Since $|E(P^0_{xy})| = 2^{n-1} - 1 > 4 + 4$, this is possible always.
 
(iv) Otherwise, select three disjoint edges from $P^0_{xy}$ are chosen, ensuring that are not incident with ${m_0, t_0, y_0}$. Since $|E(P^0_{xy})| > 8$, this is possible always.

Suppose these edges are represented by $u_0v_0, w_0z_0$ and $s_0r_0$, we ensure $u_1, v_1, w_1, z_1, s_1, r_1, m_1$, and $y$ are distinct, and $t_0 \notin \{u_0, v_0, w_0, z_0, s_0, r_0, m_0\}$. Since $\{u_1, v_1, w_1, z_1, s_1, r_1, m_1, y\}$ are balanced, by Theorem \ref{a8}, a spanning 4-path $P^1_{u_1v_1} + P^1_{w_1z_1} + P_{s_1r_1} + P^1_{m_1y}$ exists in $Q_{n-1}^1$. Hence, $P_{xy}$ in $Q_n - F$ is $P_{xy} =P^0_{xm_0} + P^1_{u_1v_1} + P^1_{w_1z_1} + P^1_{s_1r_1} + P^1_{m_1y}+ \{u_0u_1, w_0w_1, z_0z_1, v_0v_1, s_0s_1, r_0r_1, m_0m_1\}-\{u_0v_0, w_0z_0, s_0r_0\}$ is the Hamiltonian path, see in Figure \ref{Fig.6}(a).
\begin{figure}[H]
	\centering
	\includegraphics[width=0.96\linewidth, height=0.26\textheight]{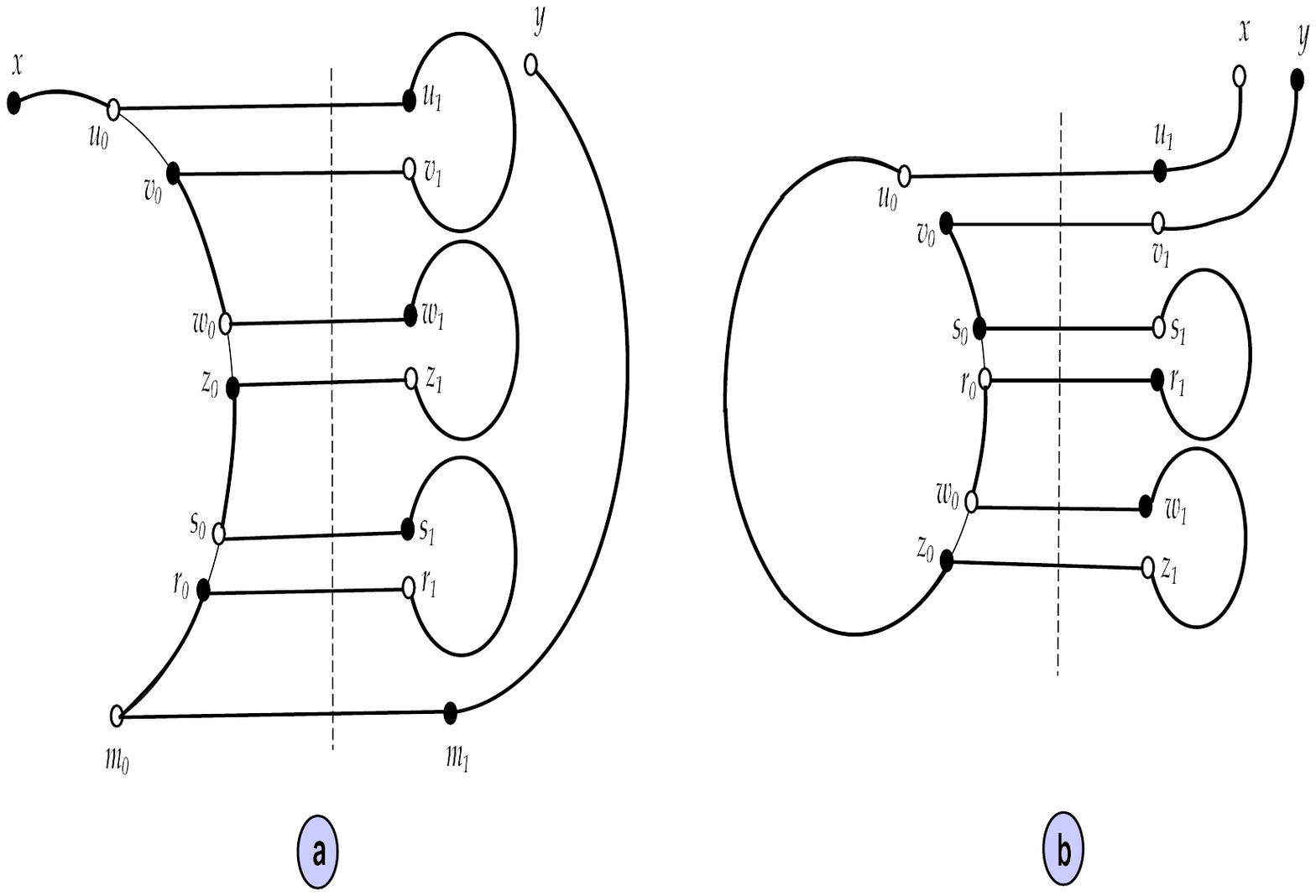}
	\caption[A]{Illustration of Subcases 4.2 and 4.3.}\label{Fig.6}
\end{figure}

Subcase 4.3. $x, y\in V(Q_{n-1}^1)$.

Given that \( |F_0| = 4n - 18  \), choose edges $\{u_0v_0, w_0z_0, s_0r_0\}\in F_0$ such that $\{u_1, v_1, w_1, z_1, s_1, r_1\}\in V(Q_{n-1}^1\setminus\{x, y\})$ and $\{u_0v_0, w_0z_0, s_0r_0\} \cap F_c = \emptyset$. Since $|F_0 \setminus \{u_0v_0, w_0z_0, s_0r_0\}| = 4(n - 1) - 17$, by induction hypothesis, a Hamiltonian path \( P_{u_0v_0}^0 \) exists in \( Q_{n-1}^0 - (F_0\setminus \{u_0v_0, w_0z_0, s_0r_0\}) \). If $\{w_0z_0, s_0r_0\} \in  E(P_{u_0v_0}^0) $, since \( p(u_1)\neq p(x) \) and $ p(v_1)\neq p(y)$, and $\{u_1, x, v_1, y, s_1, r_1, w_1, z_1\}$ are balanced, by Theorem \ref{a8}, s a spanning 4-path \( P^1_{xu_1} + P^1_{yv_1} + P^1_{w_1z_1}+ P_{s_1r_1}^1\) exists in \( Q_{n-1}^1 \). Thus, the required Hamiltonian path in \( Q_n - F \) is  $P_{xy} = P^0_{u_0v_0} + P^1_{xu_1} + P^1_{yv_1} + P^1_{w_1z_1}+ P_{s_1r_1}^1 + \{u_0u_1, v_0v_1, w_0w_1, z_0z_1, s_0s_1, r_0r_1\}-\{w_0z_0, s_0r_0\}$, see in Figure \ref{Fig.6}(b). If $w_0z_0 \notin  E(P_{u_0v_0}^0)$, we proceeds as in Subcase 3.3.

\subsection*{Data availability}
{\small No new data were generated or analyzed in support of this research.}
\subsection*{Conflict of interest}
{\small The authors have no conflict of interest.}
\subsection*{Acknowledgments} 
{\small We will express our appreciation to the anonymous referees for their productive feedback to improve the clarity of paper.}

\end{document}